\documentclass{amsart}%
\usepackage[all]{xy}
\usepackage{amssymb}
\usepackage{amsfonts,amsthm,amsmath,amscd}
\usepackage{mathrsfs}

\newcounter{zlist}

\newcounter{blist}

\newcounter{rlist}

\input xy
\xyoption{all}
\xyoption{2cell}
\UseAllTwocells

\swapnumbers
\newtheorem{theorem}{Theorem}[section]
\newtheorem{lemma}[theorem]{Lemma}
\newtheorem{thm}[theorem]{}
\newtheorem{proposition}[theorem]{Proposition}

\newtheorem{remark}[theorem]{Remark}
\numberwithin{equation}{section}

\newcommand{\T}{{\textsf{T}}}

\newcommand{\ol}{\overline}

\newcommand{\A}{\ensuremath{\mathscr{A}}}

\newcommand{\D}{\ensuremath{\mathsf{Desc}}}

\newcommand{\xr}{\xrightarrow}
\newcommand{\ra}{\ensuremath{\xymatrix@1@C=16pt{\ar[r]&}}}
\newcommand{\dra}{\ensuremath{\xymatrix@1@C=25pt{\ar@{|->}[r]&}}}
\newcommand{\mra}{\ensuremath{\xymatrix@1@C=20pt{\ar[r] |<(0.4){\object@{|}}&}}}

\usepackage{color}

\begin{document}

\title[Descent cohomology]{Descent cohomology and factorizations of groups}

\author [Bovdi, Mesablishvili] {Victor Bovdi, Bachuki Mesablishvili}
\address{Department of  Mathematical Sciences,	UAEU,  United Arab Emirates}
\email{vbovdi@gmail.com}

\address{Department of Mathematics, Faculty of Exact and Natural Sciences
of I.~Javakhishvili Tbilisi State University, Tbilisi, Georgia}
\email{bachuki.mesablishvili@tsu.ge}

\subjclass{Primary: 18D25, 20J05, 20J06, 20J15, 20M50;
Secondary: 18C05,  18C15}

\keywords{descent cohomology; group factorization; non-abelian cohomology}

\begin{abstract} The aim of the paper is to give a full classification of factorizations of groups in terms of descent cohomology (pointed) sets introduced in \cite{Me}. We show that descent cohomology  includes Serre's non-abelian  group cohomology  as a special case. This enables us to generalize Serre's theory further to include monoids.
\end{abstract}

\maketitle

\section{INTRODUCTION}  It is well known (e.g., \cite[Chapter 17, Proposition 33]{Dummit}) that for a group $X$ and
an $X$-module $B$, there is a bijection between the first cohomology group $\mathbf{H}^1(X,B)$ and the set of
$A$-conjugacy classes of complements to $B$ in $A$, where $A$ is the semidirect product $X \!\rtimes B$ and $B$ is
identified with the subgroup $\{(\textsf{1}_X,b); b\in B\}$ of $X \!\rtimes B$. (Recall that a subgroup $B'$ of $A$
is a \emph{complement} to $B$ in $A$ if each $a\in A$ can be expressed uniquely in the form $a = bb'$, where $b \in B$ and $b' \in B'$.)
The motivation and main purpose of
this paper is to generalize this result. In particular, we investigate what happens if $B$ is no longer abelian.
In this case, the semidirect product still makes sense and it is still possible to define a \emph{pointed set}
$\mathcal{H}^1(X,B)$ which naturally generalizes the group $\mathbf{H}^1(X,B)$. This theory, called the non-abelian
cohomology theory of groups, was developed by Serre (\cite{S}). We prove in Theorem \ref{non.ab.coh.gr.ad}
that $\mathcal{H}^1(X,B)$ still classifies $A$-conjugacy classes of complements to $B$ in $X \!\rtimes B$.

This result is obtained as a particular instance of a more general result (Theorem \ref{class.eq.}) which asserts
that for a group $A$ and a (not necessarily normal) subgroup $B$ of $A$ together with the canonical inclusion $\imath_B: B \to A$,
$A$-conjugacy classes of complements to $B$ in $A$ are  in a bijective correspondence with the set $\D^1(\T^l\!\!_{\imath\!_B}, B)$.
The set $\D^1(\T^l\!\!_{\imath\!_B}, B)$ is  the set of isomorphism  classes of 1-dimensional descent cocycles of a suitably defined
monad $\T^l\!\!_{\imath\!_B}$ on the category of left $B$-sets with coefficients in the left $B$-set $B$.
$\D^1(\T^l\!\!_{\imath\!_B}, B)$ is called the 1-\emph{descent cohomology (pointed) set of} $\T^l\!\!_{\imath\!_B}$
\emph{with coefficients in}  $B$. Descent cohomology (pointed) sets were introduced in \cite{Me} and, as it is shown
in Theorem \ref{non.ab.coh.gr.}, are sufficient general to include  Serre's non-abelian group cohomology. This enables
us to generalize Serre's theory further to include monoids.

Another way of stating Theorem \ref{T-compl.} is that there is a one-to-one correspondence between the set
$\D^1(\T^l\!\!_{\imath\!_B}, B)$ and the set of equivalence classes of factorizations $A=BX$, $X$ being a subgroup
of $A$ with $B\cap X=\{\textsf{1}_A\}$, of $A$ through $B$. Two factorizations $A=BX$ and $A=BX'$ are said to be \emph{equivalent}
if the subgroups $X$ and $X'$ are conjugate in $A$. Note that when $B$ is normal in $A$, these factorizations are essentially
semidirect products $A=X \!\rtimes B$ (\cite[Chapter 5, Theorem 12]{Dummit}). Thus, Theorem \ref{class.eq.} provides
a complete classification of all possible factorizations of a given group.

The outline of this paper is as follows. After recalling in Section 2 those notions and
aspects of the category of monoid actions and of descent cohomology theory of monads that will be needed,
we develop in Section 3 descent cohomology theory of monoids.

In Section 4 present our main results. We illustrate how to apply the results of the previous section to obtain
a full classification of factorizations of groups in terms of descent cohomology (pointed) sets.

We show in Section 5 that our notion of descent cohomology (pointed) sets includes, as a special case,
Serre's non-abelian  group cohomology. We close the section by giving three examples of calculating of
1-descent cohomology sets.

\section{PRELIMINARIES}

For the basic definitions of category theory we use  \cite{Mc}. The category of sets is  denoted by $\verb"Set"$.

Let $A$ be a monoid. We usually  denote the unit and the multiplication of $A$  as $\textsf{1}_A$ and $m$ respectively.
If we deal with  more than one monoid in context, we add subscripts.  Moreover, if $a$ and $b$ are elements of a monoid
$(A,m_A, \textsf{1}_A)$, we  will often write $ab$ for $m(a,b)$ (at certain  places we write $a \cdot b$ if this helps to avoid confusion).
The map which  sends each element of a monoid $A$ to the unit element $\textsf{1}_B$ of a monoid $B$ is  denoted by $0_{A,B}$.

The fact that $B$ is a (proper) submonoid (resp. subgroup) of a monoid (resp. group) $A$ is denoted by  $B \leqslant A$ (resp. $B < A$).
In this case we write $\imath_B$ for the canonical embedding $B \subseteq A$. If $B$ and $X$ are subsets of a monoid $A$,
we use the notation
\[BX= \{bx : b\in B \,\,\text{and}\,\, x \in X\}.
\]

\subsection{The category of monoid actions.} \label{actions}
Let $A=(A,m,\textsf{1})$ be a monoid. A \emph{left $A$-set} is a pair $(X, \rho_X)$ consisting of a set
$X$ and a map $\rho_X: A \times X \to X $ written as $\rho_X(a,x)=ax$ and is called the \emph{action}
(or the $A$-\emph{action}) on $X$,    if it  satisfies
\[
a(a'x)=(aa')x, \qquad \textsf{1}x=x, \qquad (\forall  a,a' \in A, \forall x \in X).
\]
The monoid $A$ is said to \emph{act on} $X$ (from the left). The set X is called a \emph{(left)} $A$-\emph{set}.
If $X$ and $Y$ are two left $A$-sets, a morphism from $X$ to $Y$ is a map  $f: X \to Y$ such that
$$
f(ax)=af(x),  \qquad \qquad (\forall a \in A,\quad \forall x \in X).
$$
The left $A$-action  and their morphisms
form a category which we denote by ${^{A}\verb"Set"}$.

Alternatively, the category $^A\verb"Set"$ may be considered as the Eilenberg-Moore category of algebras of the monad on
$\verb"Set"$ whose functor-part is the endofunctor
\[
A \times -:\verb"Set" \to \verb"Set", \qquad  X \to A \times X.
\]
Then it follows that the forgetful functor $U: {^A\verb"Set"} \to \verb"Set"$ that takes a left $A$-action $(X, \rho_X)$ to
$X$ has a left adjoint $F: \verb"Set" \to ^A\verb"Set"$ sending a set $X$ to the "free" $A$-set $(A\times X, m \times X)$.

Analogously, one has the category $\verb"Set"^A$ of right $A$-sets. $\verb"Set"^A$ can be equivalently
regarded as the Eilenberg-Moore category of algebras of the monad whose functor-part is the endofunctor
\[
- \times A: \verb"Set" \to \verb"Set", \qquad  X \to A \times X.
\]
If $B$ is another monoid, by an $(A,B)$-biset is meant a set $X$,  which is simultaneously a left $A$-set
and a right $B$-set, such that $a(xb) = (ax)b$. To indicate that $X$ is a
$(A,B)$-biset, we shall sometimes write it as $_AX_B$. Notice that any monoid $A$ may be regarded as both a
left and a right $A$-set, with actions given by the multiplication in $A$.

Let $X \in {^{A}\verb"Set"}$ be a left $A$-set. The \emph{orbit} of an element $x \in X$ is the set
\[
\textsf{Orb}_A(x)=\{ax\mid  a \in A\}\subseteq  X.
\]
Obviously, if   $A$ is a group, then  $\{\textsf{Orb}_A(x)\mid x\in X\}$  form a partition of $X$.
The quotient $X/A$ of the associated equivalence relation on $X$ is called the \emph{quotient} of the action
of $A$ on $X$.

\subsection{ Descent cohomology (pointed) sets. }\label{des.coh.}
We recall the notion of \emph{descent cohomology (pointed) set} from  \cite{Me}.

Given a monad $\T$ on a  category $\A$ and an object $a \in \A$, the set $\mathcal{Z}^{1}(\T, a)$
\emph{of descent 1-cocycles of} $\T$ \emph{with coefficients in} $a$ is defined as the set of all
$\T$-algebra structures on $a$. For any isomorphism $\sigma: a \to b$ in $\A$, one can transport a
$\T$-algebra structure $h: T(a) \to a$ on $a$ to $b$ along the isomorphism $\sigma$. The resultant
$\T$-algebra structure $h^\sigma: \T(b)\to b$ on $b$ is the composite $\T(b)\xr{\T(\sigma)}T(a)\xr{h}a \xr{\sigma^{-1}}b$.
This gives a bijection
\[
\mathcal{Z}^{1}(\T, \sigma): \mathcal{Z}^{1}(\T, a) \to \mathcal{Z}^{1}(\T, a')
\]
of sets. In particular, transporting $\T$-algebra structures along automorphisms of $a$ gives rise to a left  action
$$\ensuremath{\mathsf{Aut}}_{\A}(a) \times \mathcal{Z}^{1}(\T, a) \to \mathcal{Z}^{1}(\T, a)$$
of the group $\ensuremath{\mathsf{Aut}}_{\A}(a)$ on the set $\mathcal{Z}^{1}(\T, a)$.
The 1-\emph{descent cohomology set of} $\T$ \emph{with coefficients in} $a$,
denoted $\D^1(\T, a)$, is the quotient $\mathcal{Z}^{1}(\T, a)/\ensuremath{\mathsf{Aut}}_{\A}(a)$.
It is easy to see that  $\D^1(\T, a)$ is the set of equivalence classes of $\T$-algebra structures on $a$,
where two $\T$-algebra structures are equivalent if they are
isomorphic as objects of the Eilenberg-Moore category of $\T$-algebras, $\A^\T$.

When $a$ already carries a $\T$-algebra structure $h: T(a) \to a$, this structure makes
$\mathcal{Z}^{1}(\T, a)$ a pointed set, with the distinguished point given by $h$.
Accordingly, $\D^1(\T,a)$ is also a pointed set with base point represented by the equivalence class
$[(a,h)]$ of $(a,h)$. We shall indicate this by writing $\mathcal{Z}^{1}(\T, (a,h))$
(resp. $\D^1(\T,[(a, h)])$) rather than  $\mathcal{Z}^{1}(\T, a)$ (resp. $\D^1(\T,a)$).
Moreover, in this special case, the 0-\emph{descent cohomology group} $\D^0(\T, (a, h))$ \emph{of} $\T$ \emph{with
coefficients in} $(a, h)$ can also be defined as the group of all automorphisms of
$(a, h)$ in $\A^{\T}$. Thus $\D^0(\T, (a, h))=  \ensuremath{\mathsf{Aut}}_{\A^\T}(a, h)$.

Obviously  for each isomorphism $\sigma: a \to b$, the bijection $\mathcal{Z}^{1}(\T, \sigma)$
yields a bijection $$\D^1(\T,\sigma): \D^1(\T,a)\to \D^1(\T,b).$$ Moreover, if $(a, h) \in \A^{\T}$,
then $\mathcal{Z}^{1}(\T, \sigma): \mathcal{Z}^{1}(\T, (a,h)) \to \mathcal{Z}^{1}(\T, (b,h^\sigma))$
(resp. $\D^1(\T,\sigma): \D^1(\T, (a, h)) \to \D^1(\T, (b, h^\sigma))$) becomes a bijection of sets
(resp. an isomorphism of pointed sets).

\section { DESCENT COHOMOLOGY THEORY FOR MONOIDS}\label{mon.}

Let $A=(A,m, \textsf{1})$ be a monoid. For any right $A$-set $(X,\varrho_X)$ and a left $A$-set $(Y, \rho_Y)$,
one defines their \emph{tensor product} as the coequalizer in $\verb"Set"$,
\[{\xymatrix @R=.5in @C=.2in { X \!\times \! A \times \! Y \ar@{->}@<0.5ex>[rr]^-{\varrho_X \times Y} \ar@
{->}@<-0.5ex> [rr]_-{X \times \rho_Y}&& X \! \times \! Y \ar[r]^-{\pi}& X \! \otimes_A \!Y.}}\]
According to the coequalizer construction in $\verb"Set"$, $X \! \otimes_A \!Y$ is the quotient $(X \! \times \!Y)/\theta$,
where $\theta$ is the equivalence relation on $X \! \times \!Y$ generated by
$$
((xa,y),(x,ay)),\qquad\qquad   (x \in X, y \in Y, a \in A)
$$ and $\pi(x,y)$, which is denoted by $x\otimes_A y$, is the
$\theta$-equivalence class $[(x,y)]_\theta$ of $(x,y)$.

If $X$ is a $(B,A)$-biset  and  $Y$ is an $(A,C)$-biset, then $X \! \otimes_A \!Y$ has the structure of a   $(B,C)$-biset.
Left and right actions by $B$ and $C$ are defined by the following rules
\[
b(x,y)=(bx,y)\qquad\qquad\text{and}\qquad\qquad  (x,y)c=(x,yc).
\]
For any homomorphism of monoids $\iota: B \to A$, the natural $A$-biset structure on $A$
gives rise to a $(B, B)$-biset structure on $A$ via $\iota$. $A$ then receives the natural $(A, B)$-
biset structure and thus for any left $B$-set $X$, $A \otimes_B X$ is a left $A$-set via the action
defined by $a(a' \otimes_A x)=aa' \otimes_A x$ for all $a,a' \in A$ and $x \in X$. This extends
naturally to a functor
\[
\iota_*=A\!\otimes_{B} -: {^{B}\verb"Set"} \to {^{A}\verb"Set"},\qquad X\longmapsto A\!\otimes_{B}\! X,
\]
which one calls the \emph{extension-of-scalars functor}. Moreover, any  left $A$-set $(Y,\varrho)$ can be
turned into a left $B$-set with a left action defined by the composition
\[
\widehat{\varrho}: B \times Y \xr{\iota \times Y} A\times Y \xr{\varrho}Y,
\]
and one obtains in this way a functor
\[\iota^*: {^{A}\verb"Set"} \to {^{B}\verb"Set"}, \]
called \emph{restriction-of-scalars}. It is well-known that the restriction-of-scalars functor is right adjoint
to the extension-of-scalars functor. Write $\T^l\!\!_\iota=(T_\iota, \mu_\iota, \eta_\iota)$ for the monad on
${^{B}\verb"Set"}$ generated by this adjunction.
Then
\begin{itemize}
  \item $T^l\!\!_\iota(X, \rho_X)=(A \!\otimes_{B} \!\!X, B\!\times \!(A \!\otimes_{B} \!X) \xr{(b,a \otimes_{B}x)
  \mapsto (\iota(b)a \otimes_{B}x)}  A\!\otimes_{B}\!X) $;
  \item $(m_\iota)_{(X, \rho_X)}: A\!\otimes_{B} \!A \!\otimes_{B} \!X \xr{\bar{m}\otimes_{B}X}A\!\otimes_{B}\!X$;
  \item $(\eta_\iota)_{(X, \rho_X)}: X \xr{x \mapsto \textsf{1}\otimes_{B} \,x}A\!\otimes_{B}\!\!X$,
\end{itemize} for all left $B$-set $(X, \rho_X)$. Here $\bar{m}: A\!\otimes_{B} \!A \to A$ is the unique morphism
making the diagram
\begin{equation} \label{m}\xymatrix @R=.3in @C=.4in{
A \times A \ar[r]^-{\pi} \ar[rd]_{m}&A\!\otimes_{B} \!A \ar[d]^{\bar{m}}\\
&A}
\end{equation}
commutative. In other words, ${\bar{m}(a\otimes_Ba')=aa'}$.

Symmetrically, viewing $A$ as an $(X, A$)-biset via $\iota$, one has the functor
\[
\iota_\star=-\otimes_{B}A: {\verb"Set"}^B \to {\verb"Set"}^A,
\]
its left adjoint $$\iota^\star: {\verb"Set"}^A \to {\verb"Set"}^B$$
and the corresponding monad $\T^r\!\!\!_\iota$ on ${\verb"Set"}^{B}$.

Since $(B, m_B)\in {^{B}\verb"Set"}$, we can consider  $\T^l\!\!_\iota$-algebra structures on $(B, m_B)$.
The proof of the following theorem is straightforward and is omitted.

\begin{theorem} \label{T-str}
Let  $\iota: B \to A$ be a  morphism  of monoids.  The assignment
\[
 (A\!\otimes_{B} \!B  \xr{h}B) \longmapsto (A \xr[\simeq] {a \to  a\otimes_{B}\textsf{1}_B} A\!\otimes_{B} \!B \xr{h} B)
\]
yields a one-to-one bijection between
$\mathcal{Z}^{1}(\emph{\T}^l\!\!_\iota,  (B,m_B))$ and the set of maps $q: A \to B$ such that
\begin{enumerate}
  \item [(ZL1)]  $q(\emph{\textsf{1}}_{A})=\emph{\textsf{1}}_B$;
  \item [(ZL2)] $q(\iota(b) a)=b  q(a)$;
  \item [(ZL3)] $q(aa')=q(a\cdot \iota q(a'))$,\quad \qquad  ($\forall a, a' \in A$, $b\in B$).
\end{enumerate} Under this bijection, $q: A \to B$ corresponds to the composite
\[
A\!\otimes_{B} \!B \xr[\simeq]{a\otimes_{B}b \to  a\cdot \iota(b)} A \xr{q}B.
\]
\end{theorem}

\begin{thm}{\bf Remark.} \label{T-str.r}\em
It is easy to see that under (ZL2), (ZL1) is equivalent to the condition  $q \iota =\textsf {Id}_B$.
\end{thm}

Symmetrically  we have the following.

\begin{theorem} \label{T-str.1.d.}Given a morphism $\iota: B \to A$ of monoids, there is a one-to-one bijection between
$\mathcal{Z}^{1}(\emph{\T}^r\!\!_\iota,  (B,m_B))$ and the set of maps $p: A \to B$ such that
\begin{enumerate}
  \item [(ZR1)] $p(\emph{\textsf{1}}_{A})=\emph{\textsf{1}}_B$;
  \item [(ZR2)] $p(a \iota(b))= p(a)b$;
  \item [(ZR3)] $p(aa')=p(\iota p(a) a')$, \qquad ($\forall a, a' \in A, b\in B$).
\end{enumerate}
\end{theorem}

\begin{thm}{\bf Remark.} \label{remark.str}\em
\begin{itemize}
  \item [(a)] It can be easily  verified that if $r: A \to B$ is a monoid homomorphism with $r \iota =\textsf{Id}_B$,
  then $r \in \mathcal{Z}^{1}(\T^l\!\!_\iota,  (B,m_B))$ and $r \in \mathcal{Z}^{r}(\T^l\!\!_\iota,  (B,m_B))$.
\item [(b)]  Since (see Theorem \ref{T-str})
\[
\begin{split}
q(a)=q(\textsf{1}_A a)\stackrel{(ZL3)}=&
  q(\textsf{1}_A \cdot \iota q(a))\\
  =&q(\iota q(a)\cdot \textsf{1}_A)\stackrel{(ZL2)}=q(a) q(\textsf{1}_A),\qquad (a \in A)
\end{split}
\]
under Conditions (ZL2) and (ZL3) the equality $q(\textsf{1}_{A})=\textsf{1}_B$ holds, provided
  that there is an element $a \in A$ whose image under $q$ is invertible. In particular, this  is always the case when  $B$ is a group.
  \end{itemize}

\bigskip

\noindent Combining Remark \ref{remark.str}(b) and its dual with Proposition \ref{T-str} gives

\begin{proposition} \label{T-str.00}
Let $\iota: B \to A$ be a morphism of monoids, where $B$ is a group. Then elements of the set
$\mathcal{Z}^{1}(\emph{\T}^l\!\!_\iota,  (B,m_B))$ (resp. $\mathcal{Z}^{1}(\emph{\T}^r\!\!_\iota,  (B,m_B))$) are those maps
$A \to B$ which satisfy Conditions (\emph{ZL2}) and (\emph{ZL3}) (resp. (\emph{ZR2}) and (\emph{ZR3})).
\end{proposition}

\begin{proposition} \label{subgr.} Let $\iota: B \to A$ be a morphism of monoids, where $A$ is a group. If
either $\mathcal{Z}^{1}(\emph{\T}^l\!\!_\iota,  (B,m_B))$ or  $\mathcal{Z}^{r}(\emph{\T}^l\!\!_\iota,  (B,m_B))$ is non-empty, then $B$ is also a group.
\end{proposition}
\begin{proof} If $\mathcal{Z}^{1}(\T^l\!\!_\iota,  (B,m_B))\neq\varnothing$ and $q \in \mathcal{Z}^{1}(\T^l\!\!_\iota,  (B,m_B))$,
then $\iota$ is injective by Remark \ref{T-str.r}. Therefore, we can suppose that $B$ is a submonoid of $A$ and  $\iota$ is the canonical inclusion.
To show that any element $b\in B$ is invertible in $B$, it is enough to show that the inverse $b^{-1}$  of $b$ in $A$ lies in $B$.
But since $\textsf{1}_A=q(\textsf{1}_A)$ by (ZL1) and since $bb^{-1}=\textsf{1}_A$, it follows that
\[
\textsf{1}_A=q(\textsf{1}_A)=q(bb^{-1})\stackrel{(ZL2)}=bq(b^{-1}).
\]
Consequently, $q(b^{-1})=b^{-1}$ in $A$ and hence $b^{-1}\in B$.
The  $\mathcal{Z}^{1}(\T^r\!\!_\iota,  (B,m_B))\neq\varnothing$ case can be treated  similarly.
\end{proof}

\begin{proposition}\label{iso-T-str}
If the conditions  of Theorem \ref{T-str} holds, then two $\emph{\T}^l\!\!_\iota$-algebra structures
$q, \, q\,': A \to B$ on $(B, m_B)$ are isomorphic if and only if there is an invertible element $b_0 \in B$ such that
$q(a) b_0=q\,'(a\cdot  \iota(b_0))$ for all $a \in A$.
\end{proposition}
\begin{proof} Two $\T^l\!\!_\iota$-algebra structures $q, q\,': A \to B$ on $(B, m_B)$
are isomorphic if and only if there is an isomorphism $f: B \to B$ of left $B$-sets making the diagram
\[
\xymatrix @R=.4in @C=.6in{
A\!\otimes_{B} \!B \ar[r]^{r\!_A}\ar[d]|{A\otimes_{B} \!f}& A\ar[d]|{r\!_A \cdot (A\otimes_{B} \!f)
\cdot r\!_A^{-1}} \ar[r]^-{q}& B\ar[d]^f\\
A\!\otimes_{B} \!B \ar[r]_{r\!_A}&A \ar[r]_-{q'}& B,}
\] where $r_A$ is the isomorphism $A\!\otimes_{B} \!B \xr{a\otimes_{B}b \to  a\cdot \iota(b)} A$,
commute. Since $f$ is a morphism of left $B$-set, there exists an invertible element $b_0 \in B$
(namely, $f(\textsf{1}_B)$\footnote{To see that the element $f(\textsf{1}_B)$  is indeed invertible
in $B$, we note first that since $f$ is a morphism of left $B$-sets, $f(b)=bf(\textsf{1}_B)$ for all
$b\in B$.  It follows -- since $f$ is bijective -- that $f(b_0)=b_0f(\textsf{1}_B)=\textsf{1}_B$ for
some $b_0 \in B$. Next, since $f (f(\textsf{1}_B)b_0)=f(\textsf{1}_B)f(b_0)=f(\textsf{1}_B)$ and
since $f$ is bijective, one concludes that $f(\textsf{1}_B)b_0=\textsf{1}_B$, proving that
$f(\textsf{1}_B)$  is invertible with inverse $b_0$.}) such that $f(b)=b b_0$ for all $b \in B$.
Therefore, the last diagram commutes if and only if
\begin{equation}\label{equivalent}q(a)b_0=q\,'(a  \cdot\iota(b_0)), \qquad\quad  (\forall a \in A).
\end{equation}
\end{proof}
Particular examples of maps satisfying Conditions (ZL1)--(ZL3) arise naturally from Schreier split
extensions of monoids as exhibited in the following

\begin{proposition}\label{schreier} Let $(A , \,B ,\, p , \,j)$ be a Schreier split epimorphism of
monoids in the sense of \cite[Definition 2.1.2.]{Bourn} and let
\[\xymatrix @R=.4in @C=.6in{\emph{\textsf{Ker}}(p)  \ar[r]_\kappa  &A \ar@/_1.6pc/@{-->} [l]_{q}\ar@{->}@<+0.8ex>[r]^-{p}
 &  B\ar@<+0.6ex>[l]^{j}}
\]
be the canonical Schreier split sequence associated with it. Then $q: A \to \emph{\textsf{Ker}}(p)$ defines a
$\emph{\T}^l\!\!\!_\kappa$-algebra structure on $\emph{\textsf{Ker}}(p)$ (or, equivalently, $q \in \mathcal{Z}^{1}(\emph{\T}^l\!\!\!_\kappa,
(\emph{\textsf{Ker}}(p),m_{\emph{\textsf{Ker}}(p)}))$.
\end{proposition}
\begin{proof} According to Proposition \ref{T-str} and Remark \ref{T-str.r}, $q$ defines a $\T^l\!\!\!_\kappa$-algebra
structure on $\textsf{Ker}(p)$ provided that
\begin{enumerate}
  \item $q  \kappa=\textsf{Id}_{\textsf{Ker}(p)}$;
  \item $q$ is a morphism of left $\textsf{Ker}(p)$-sets;
  \item $q(a \cdot a')=q(a \cdot \kappa q(a'))$.
\end{enumerate}

Since (1) is just \cite[Proposition 2.1.5.(a)]{Bourn}, it remains to prove (2) and (3).

\bigskip

\noindent   To show (2), note first that for each $x \in \textsf{Ker}(p)$ and  $a \in A$, we have
$$jp(\kappa(x) a)=jp(\kappa(x))\cdot jp(a)=j(\textsf{1}_B) \cdot jp(a)=\textsf{1}_A \cdot jp(a)=jp(a).$$ In the first equality, we use the fact that both
$p$ and $j$ are morphisms of monoids. The second equality holds because $x \in \textsf{Ker}(p)$ and $(\textsf{Ker}(p), \kappa)$ is the kernel of $p$. Thus,
\begin{equation}\label{kernel}
jp(\kappa(x) a)=jp(a),  \qquad\qquad (\forall   x \in K[p],  a \in A).
\end{equation} According to \cite[Proposition 2.1.4.]{Bourn}, we have
 \begin{equation}\label{kernel.1} \kappa(x)a=q(\kappa(x) a)\cdot jp(\kappa(x)a)\end{equation}
and
\begin{equation}\label{kernel.2} a=\kappa q(a)\cdot jp(a).\end{equation}
But
$$\kappa(x) a \stackrel{(\ref{kernel.2})}=\kappa(x) \cdot \kappa q(a)\cdot jp(a)\stackrel{(\ref{kernel})}=\kappa(x)\cdot q(a)\cdot jp(x\cdot a)$$
and comparing the last equality and (\ref{kernel.1}) yields  by the uniqueness of $q$ that
\begin{equation}\label{kernel.3} q(\kappa(x)a)=x\, q(a).
\end{equation}
Hence $q$ is a morphism of left $\textsf{Ker}(p)$-sets.

As for (3), since $q(a a') = q(a) \cdot q(jp(a) \cdot \kappa q(a'))$ by \cite[Proposition 2.1.5.(e)]{Bourn},
\[
\begin{split}
q(a a') &= q(a) \cdot q(jp(a) \cdot \kappa q(a'))
\stackrel{(\ref{kernel.3})}= \\
&=q(\kappa q(a) \cdot jp(a)\cdot \kappa q(a'))  \stackrel{(\ref{kernel.2})}=q(a \cdot \kappa q(a')).
\end{split}
\]
This completes the proof.
\end{proof}

The \emph{center} $Z(A)$ of a  monoid $A$ is the set of those elements of $A$ which commute with all elements of $A$;
clearly, $\textsf{1}_A \in Z(A)$. A \emph{central submonoid }of a monoid $A$ is a submonoid $X$ of $A$ such that
$X \subseteq Z(A)$. A morphism $\iota: B \to A$ of monoids is called \emph{central} if $B$ is commutative  and its image
$\iota(B)\subseteq Z(A)$,  that is $$a \in A, b \in B \Rightarrow \iota(b)a=a\iota(b).$$

\begin{proposition}\label{central}  Let  $\iota: B \to A$ be a central homomorphism of monoids. Then
$$\mathcal{Z}^{1}(\emph{\T}^l\!\!_\iota,  (B,m_B))=\mathcal{H}^{1}(\emph{\T}^l\!\!_\iota,  (B,m_B)).$$ Moreover,
a map $q: A \to B$ lies in $\mathcal{Z}^{1}(\emph{\T}^l\!\!_\iota,  (B,m_B))$ if and only
it is a retraction of $\iota$ (i.e., a monoid homomorphism with $q \iota =\emph{\textsf{Id}}_B$).
\end{proposition}
\begin{proof} We first prove that $q \in \mathcal{Z}^{1}(\T^l\!\!_\iota,  (B,m_B))$ (i.e. $q: A \to B$ defines a
$\T^l\!\!_\iota$-algebra structure on $B$) if and only if $q$ is a homomorphism of monoids with $q \iota =\textsf{Id}_B$.
Since  the forward implication follows from Remark \ref{remark.str}(a), suppose that $q \in \mathcal{Z}^{1}(\emph{\T}^l\!\!_\iota,  (B,m_B))$.
Then  we have  that
\begin{align*}
q(aa')&=q(a\cdot \iota q(a'))& \text{by }\quad (\emph{ZL3})\\
&=q(\iota q(a')a)& \text{by centrality of}\quad \iota\\
&=q(a')q(a)& \text{by}\quad  (\emph{ZL2})\\
&=q(a)q(a'), & \text{by commutativity of}\quad B
\end{align*}
for all $a,a' \in A$. Thus, $q(aa')=q(a)q(a')$ and since $q(\textsf{1}_A)=\textsf{1}_B$ by (ZL1), it follows that $q$ is a homomorphism of monoids.

Let $q, q\,': A \to B$ be two $\T^l\!\!_\iota$-algebra structures  on $(B, m_B)$.
Classes $[q]$ and $[q']$ coincide in $\mathcal{H}^{1}(\T^l\!\!_\iota,  (B,m_B))$ (see  Proposition \ref{iso-T-str})
if and only if there exists  an invertible element $b_0 \in B$ such that $q(a) b_0=q\,'(a\cdot  \iota(b_0))$ for all $a \in A$.
As we have just seen, $q'$ is a homomorphism of monoids, so
\[
q\,'(a\cdot  \iota(b_0))=q'(a)q'(\iota(b_0))=q'(a) b_0.
\]
 Hence
$q(a) b_0=q'(a) b_0$ for all $a \in A$ and $q=q'$, because  $b_0$ is invertible. This  proves  the equality of sets
$\mathcal{Z}^{1}(\T^l\!\!_\iota,  (B,m_B))$ and $\mathcal{H}^{1}(\T^l\!\!_\iota,  (B,m_B))$.
\end{proof}

\end{thm}

\section{THE CASE OF GROUPS}\label{grp.}

Let $B$ be a subgroup of a group $A$. Recall that  a subgroup $X$ of $A$ is called a
\emph{complement} to the subgroup $B$ in $A$ if $A=BX$ and $B\cap X=\{\textsf{1}_A\}$.
Recall also that if  $B$ is a subgroup of a group $A$  and $X$ is a complement to $B$ in $A$,
then each element $a \in A$ is uniquely expressible in the form
\[
a=b_ax_a,  \qquad (b_a \in B, x_a \in A).
\]
Recall further that a group $A$ is said to be \emph{factorizable} if $A = BX$
for some proper subgroups $B$ and $X$ of $A$. The expression $A = BX$ is called a \emph{factorization} of $A$.

\begin{lemma}\label{ccj}
Let $X$ be  a complement to $B$ in $A$. Then each  $A$-conjugate of $X$ is also a complement for $B$ in $A$.
\end{lemma}
\begin{proof}Let $a \in A$ and let $X$ be a complement to $B$ in $A$. Since $X$ is a complement to $B$ in $A$,
there exist  $b_a \in B$ and $x_a\in X$ such that $a=b_ax_a$, so
\[
aXa^{-1}=b_ax_a X a_a^{-1}b_a^{-1}=b_aXb_a^{-1}.
\]
Thus, $A$-conjugates for $X$ in
$A$ are in fact $B$-conjugates. Now, if $b=b_axb_a^{-1}$ for some $x \in X$, then $x=b_a^{-1}b b_a$
and hence $x \in B$. But since $B\cap X=\{\textsf{1}_A\}$, it follows that $x=\textsf{1}_A$, and hence $b=\textsf{1}_A$.
This yields  $B\cap b_aXb_a^{-1}=\{\textsf{1}_A\}$. Moreover, since $BX=A$, we have
$$
Bb_aXb_a^{-1}=BXb_a^{-1}=Ab_a^{-1}=A,
$$
so  $aXa^{-1}=b_aXb_a^{-1}$ is a complement to $B$ in $A$.
\end{proof}

We henceforth suppose that $\iota: B \to A$ is  a morphism of groups and
write $\mathcal{C}_\iota$ for the set of the complements to $\iota(B)$   in $A$.

\begin{proposition}\label{T-compl.0} For each  $q \in \mathcal{Z}^{1}(\emph{\T}^l\!\!_\iota,  (B,m_B))$, the set
\[
\emph{\textsf{Ker}} (q)=\{a\in A: q(a)=\textsf{1}_B\}
\]
 is a complement to $\iota(B)$ in $A$.
\end{proposition}
\begin{proof} Let $q \in \mathcal{Z}^{1}(\T^l\!\!_\iota,  (X,m_X))$, or, equivalently, let $q: A \to X$ be a map
satisfying Conditions (ZL1), (ZL2) and (ZL3). We first show that $\textsf{Ker}(q)$ is a subgroup of $A$. Indeed,
since $q(\textsf{1}_A)=\textsf{1}_B$ by (ZL1), it follows that $\textsf{1}_A \in \textsf{Ker}(q)$. Next, if
$a,a' \in \textsf{Ker}(q)$, then $q(a)=q(a')=\textsf{1}_B$ and we calculate
\[
q(a a')\stackrel{(ZL3)}=q(a \cdot \iota q(a'))=q(a \cdot \iota
(\textsf{1}_B))=q(a\textsf{1}_A)=q(a)=\textsf{1}_B.
\]
 Hence $aa' \in \textsf{Ker}(q)$.
Finally, if $a \in \textsf{Ker}(q)$, then $q(a)=\textsf{1}_B$ and we have
\[
\begin{split}
\textsf{1}_B \stackrel{(ZL1)}=q(\textsf{1}_A)= &q(a^{-1} a) \stackrel{(ZL3)}= q(a^{-1}\cdot \iota q(a))\\
=&
q(a^{-1}\cdot \iota(\textsf{1}_X))=q(a^{-1} \textsf{1}_A)=q(a^{-1}),
\end{split}
\]
 whence $a^{-1} \in \textsf{Ker}(q)$.
Thus, $\textsf{Ker}(q)$ is a subgroup of $A$.

Next we show that $\iota(B)\cap\textsf{Ker}(q)=\{\textsf{1}_A\}$ and $\iota(B)\textsf{Ker}(q)=A$.
In order to show the first equality, consider an arbitrary element $a \in \iota(B)\cap\textsf{Ker}(q)$.
Then $q(a)=\textsf{1}_B$ and $a=\iota(b)$ for some $b \in B$. But then
$b \stackrel{(ZL1)}=q\iota(b)=q(a)=\textsf{1}_B$ and hence $a=\iota(b)=\iota(\textsf{1}_B)=\textsf{1}_A$.
Therefore,
\[
\iota(B)\cap\textsf{Ker}(q)=\{\textsf{1}_A\}.
\]
We now turn to the second equality. Note first that for any $a\in A$, $\iota(q(a)^{-1}) a\in
\textsf{Ker}(q)$, since $q(\iota(q(a)^{-1}) a)\stackrel{(ZL2)}=q(a)^{-1}  q(a)=\textsf{1}_A$. It follows
that any $a \in A$ can be written in the form $\iota q(a) \cdot ((\iota(q(a)^{-1}) a))$ for $\iota q(a) \in \iota(B)$
and $(\iota(q(a)^{-1})\cdot a) \in \textsf{Ker}(q)$.
This shows that $\iota(B)\textsf{Ker}(q)=A$, and hence $\textsf{Ker}(q) \in \mathcal{C}\!_\iota$.
\end{proof}

\begin{proposition}\label{T-compl.01}
Let $\iota: B \to A$ be a monomorphism of groups. Then the assignment $q \longmapsto \emph{\textsf{Ker}}(q)$
yields a one-to-one bijection between the sets $\mathcal{Z}^{1}(\emph{\T}^l\!\!_\iota,  (B,m_B))$ and
$ \mathcal{C}\!_\iota$.
\end{proposition}
\begin{proof} According to Proposition \ref{T-compl.0}, the assignment $q \longmapsto \textsf{Ker}(q)$ yields a map
$\textsf{Ker}: \mathcal{Z}^{1}(\T^l\!\!_\iota,  (B,m_B))\to \mathcal{C}\!_\iota$.
For any complement $X \subseteq A$ to $\iota(B)$ in $A$, define a map $q_X: A \to B$ by putting $q_X(a)=b_a$, where
$b_a\in B$ is such that $a$ can be written in the form $ \iota(b_a) x_a$ for a unique $x_a \in X$.
We claim that the map $q_X: A \to B$ satisfies Conditions (ZL1), (ZL2) and (ZL3).

Since $\iota(b)=\iota(b) \cdot \textsf{1}_A$ for all $b \in B$, it follows that $q_X(\iota(b))=b$. Then, in particular,
$q_X(\textsf{1}_A)=q_X(\iota(\textsf{1}_B))=\textsf{1}_B$.  Thus Condition (ZL1) is satisfied. Condition (ZL2)
is also satisfied, as the following calculation shows
\[
q_X(\iota(b)\,a)=q_X(\iota(b)\, \iota(b_a)\, x_a)=q_X(\iota(b \, b_a)\, x_a)=b\, b_a=b \, q_X(a)
\]

Next, if $a=\iota(b_a)\, x_a$ and $a'=\iota(b_{a'})\, x_{a'}$, then $q_X(a')=x_{a'}$ and we have
\begin{equation}\label{sb}q_X(a \cdot \iota q_X(a'))=q_X(a \, \iota(b_{a'})). \end{equation}
Moreover, if  $aa'=\iota(b_{a a'})\, x_{aa'}$, then $a \, \iota(b_{a'})\, x_{a'}=\iota(b_{a \,a'}) x_{a \, a'}$. Hence
$$
a \iota(b_{a'})=\iota(b_{a \, a'}) (x_{a \, a'}\, x_{a'}^{-1}),
$$
implies that $q_X(a \,\iota(b_{a'}))=b_{a a'}=q_X(a \,a')$ since clearly $x_{a \,a'} x_{a'}^{-1}\in X$. Then it follows from (\ref{sb}) that $q_X(a \,a')=q_X(a \cdot \iota q_X(a'))$.
This yields  Condition (ZL3).

Now, let us prove  that
\[
\textsf{Ker} (q_X)=X\quad \text{and}\quad q_{\,\textsf{Ker}(q)}=q,\qquad   \big(\forall X \in \mathcal{C}\!_\iota,\quad  q\in \mathcal{Z}^{1}(\T^l\!\!_\iota,  (B,m_B))\big).
\]
Obviously $X\ni x=\textsf{1}_A  x=\iota (\textsf{1}_B)\, x $, so $q_X(x)=\textsf{1}_B$ and
$X \subseteq \textsf{Ker}(q_X)$. Conversely, if $a \in A$ is such that $q_X(a)=\textsf{1}_B$, then $a=\iota (\textsf{1}_B) \, x_a$
for some $x_a \in X$, so
\[
a=\iota (\textsf{1}_B) \, x_a=\textsf{1}_A \, x_a=x_a \in X.
\]
Hence $\textsf{Ker}(q_X) \subseteq X$.
This proves that $\textsf{Ker}(q_X)=X$.

To show the equality $q_{\,\textsf{Ker}(q)}=q$, consider an arbitrary $q\in \mathcal{Z}^{1}(\T^l\!\!_\iota,  (B,m_B))$.
Then $\textsf{Ker}(q) \in \mathcal{C}\!_\iota$ and hence any $a \in A$ can be written in the form $\iota(b_a)\, x_a$ with
$b_a \in B$ and $x_a \in \textsf{Ker}(q)$. Then $$q(a)=q(\iota(b_a)\, x_a)=b_a  q(x_a)=b_a \, \textsf{1}_B=b_a=q_{\textsf{Ker}(q)}(a).$$
This completes the proof.
\end{proof}

The following is a particular case of \cite[Chapter 5, Theorem 12]{Dummit}.

\begin{theorem}\label{dumt} Let $X,B$ be subgroups of a group $A$ such that $X$ is a complement to $B$ in $A$. If $B$ is normal in $A$,
then $A$ is (isomorphic to) a semidirect product of $B$ and $X$.
\end{theorem}

Combining this theorem with Proposition \ref{T-compl.01}, we obtain
\begin{proposition}\label{dummit} Let $\iota: B \to A$ be a monomorphism of groups such that $\iota(B)$ is normal in $A$. Then
any $\emph{\T}^l\!\!_\iota$-algebra structure on $(B,m_B)$ is of the form $B \rtimes X \xr{p_B} B$ for some (left)
semi-direct presentation $A=B \rtimes X$ of $A$.
\end{proposition}

\begin{proposition}\label{T-compl.02} Let $\iota: B \to A$ be a monomorphism of groups.
Two $\emph{\T}^l\!\!_\iota$-algebra structures on $(B,m_B)$ are isomorphic if and only if
their corresponding elements of $\mathcal{C}\!_\iota$ are conjugate in $A$.
\end{proposition}
\begin{proof} By Proposition \ref{iso-T-str}, two $\T^l\!\!_\iota$-algebra structures $q, q'$ on $(B,m_B)$ are isomorphic
if and only if there is an element $b_0 \in U(B)$ such that the diagram
$$\xymatrix @R=.4in @C=.6in{ A \ar[d]_{-\cdot \iota (b_0)} \ar[r]^-{q}& B\ar[d]^{-\cdot b_0}\\
A \ar[r]_-{q'}& B} $$  commutes, or equivalently,

\begin{equation}\label{sbb}
q(a)\, b_0=q'(a\cdot \iota (b_0)), \qquad\quad   (\forall  a \in A).
\end{equation}

\noindent
Assume  that $q, q'$ are isomorphic $\T^l\!\!_\iota$-algebra structures on $(B,m_B)$. Obviously,
\begin{align*}
a \in \textsf{Ker}(q)\Longleftrightarrow q(a)=\textsf{1}_B&\\
&\Longleftrightarrow q(a)\, b_0=b_0 & \text{since}\quad  b_0 \in U(B)\\
&\Longleftrightarrow q'(a\cdot \iota (b_0))=b_0  & \text{by}\quad (\ref{sbb})\\
&\Longleftrightarrow  b_0^{-1}\, q'(a\cdot \iota (b_0))=\textsf{1}_B & \\
&\Longleftrightarrow q'(\iota(b_0)^{-1}\cdot a\cdot \iota (b_0))=\textsf{1}_B & \text{by}\quad  (\emph{ZL2})\\
&\Longleftrightarrow \iota(b_0)^{-1}\cdot a\cdot \, (b_0) \in \textsf{Ker}(q') & \\
&\Longleftrightarrow a \in \iota (b_0) \cdot \textsf{Ker}(q') \cdot \iota(b_0)^{-1}.
\end{align*}
Thus, $\textsf{Ker}(q) = \iota(b_0)\cdot \textsf{Ker}(q') \cdot \iota(b_0)^{-1}$.

Conversely, suppose  $\textsf{Ker} (q') = \iota(b_0)^{-1}\cdot \textsf{Ker}(q) \cdot \iota(b_0)$.  For each  $x \in A$, we have
\begin{align*}
x \in  \textsf{Ker}(q)
&\Longleftrightarrow \iota(b_0)^{-1}\cdot x \cdot  \iota(b_0)\in \textsf{Ker} (q')\\
&\Longleftrightarrow  q'(\iota(b_0)^{-1}\, x \, \iota(b_0))=\textsf{1}_B\\
&\Longleftrightarrow b_0^{-1}\, q'( x \cdot  \iota(b_0))=\textsf{1}_B  & \text{by}\quad  (\emph{ZL2})\\
&\Longleftrightarrow  q'( x \cdot  \iota(b_0))=b_0.
\end{align*}
Thus \begin{equation}\label{qq}
         q'( x \cdot  \iota(b_0))=b_0,  \qquad \qquad (\forall x \in  \textsf{Ker}(q)).
       \end{equation}

Since $\textsf{Ker}(q)$ is a complement to $\iota(B)$ in $A$, each  $a \in A$ can be expressed uniquely in the form
$a=\iota(b_a)\, y_a$, where  $b_a\in B$ and $y_a\in \textsf{Ker}(q)$.  Then for all $a \in A$,
$$q'(a \cdot \iota(b_0))= q'( \iota(b_a)\, y_a \, \iota(b_0))\stackrel{(ZL2)}=b_a\, q'( y_a \cdot \iota(b_0))\stackrel{(\ref{qq})}=b_a\, b_0=q(a)\,b_0.$$
Consequently, $q, q'$ are isomorphic $\T^l\!\!_\iota$-algebra structures on $(B,m_B)$, as desired.
\end{proof}

To each  left $B$-set $X$ we can  naturally associate the \emph{action groupoid} $X/\!/B$, whose objects are elements of $X$, and whose morphisms from $X\ni x\mapsto x'\in X$ are the elements  $b \in B$ such that   $bx = x'$. Clearly,
\begin{itemize}
  \item the group $\ensuremath{\mathsf{Aut}}_{^{B}\verb"Set"}(B,m_B)$ acts on $\mathcal{Z}^{1}(\T^l\!\!_\iota,  (B,m_B))$ (see, \ref{des.coh.});
  \item $\ensuremath{\mathsf{Aut}}_{^{B}\verb"Set"}(B,m_B)\simeq B$, and,
  \item $B$ acts by conjugation on $\mathcal{C}_\iota$, since every $A$-conjugate of $\iota(B) \in \mathcal{C}_\iota$
  (which is in fact a $B$-conjugate) is also a complement to $\iota(B)$ in $A$ (see Lemma \ref{ccj} and its proof),
\end{itemize}
so we have the groupoids  $\mathcal{Z}^{1}(\T\!_\iota,  (B,m_B)) /\!/B$ and  $\mathcal{C}\!_\iota /\!/ B $. By  combining Propositions \ref{T-compl.01} and \ref{T-compl.02}, we obtain the following.

\begin{theorem}\label{T-compl.} Let $\iota: B\to A$ be a morphism of groups and let $\emph{\T}^l\!\!_\iota$ be the induced monad on
${^{B}\verb"Set"}$. The assignment $$q \longmapsto \emph{\textsf{Ker}}(q)=\{a\in A: q(a)=\emph{\textsf{1}}_B\}$$ yields an isomorphism of groupoids $ \mathcal{Z}^{1}(\emph{\T}^l\!\!_\iota,  (B,m_B))/\!/B\simeq \mathcal{C}\!_\iota /\!/ B$.  Moreover, this isomorphism induces a bijection of sets
$$\D^1(\emph{\T}^l\!\!_\iota,(B, m_B)) \simeq \pi_0(\mathcal{C}_\iota/\!/B),$$ in which   $\pi_0(\mathcal{C}_\iota/\!/B)$ is the set of
connected components of the groupoid $\mathcal{C}_\iota/\!/B$.
\end{theorem}

Given a group $A$, we write $\textsf{FAC}(A)$ for the set of all factorizations of $A$, i.e.
\[ \textsf{FAC}(A)=\{(B,X) : B < A, X < B, B\cap X=\{\textsf{1}_A\} \,\,\text{and}\,\, BX=A \}.\]
It is easy to see that
\[
\textsf{FAC}(A)=\bigsqcup_{B < A}\mathcal{C}_{\imath \!_B}.
\] It then follows from Proposition \ref{T-compl.01} that

\begin{theorem}\label{class.} For any group $A$, the assignment $q \longmapsto \emph{\textsf{Ker}}(q)$ yields a one-to-one
bijection
\[
\bigsqcup_{B < A}\mathcal{Z}^{1}(\emph{\T}^l\!\!_{\imath\!_B},  (B,m_{B})) \simeq \emph{\textsf{FAC}}(A)
\]  of sets.
\end{theorem}

Define a relation $\cong$ on the set $\textsf{FAC}(A)$ by $(B,X)\!\cong\! (B',X')$ if and only if
$B=B'$ and $X$ and $X'$ are conjugate in $A$. One  easily checks that $\cong$
is an equivalence relation on the set $\textsf{FAC}(A)$. We write $\textsf{Fac}(A)$ for the quotient set $\textsf{FAC}(A)/\!\cong$.
Then it follows by the very definition of the sets
\[\pi_0(\mathcal{C}_{\imath \!_B}/\!/B), B < A,
\]
that
\[
\textsf{Fac}(A)=\bigsqcup_{B < A}\pi_0(\mathcal{C}_{\imath \!_B}/\!/B).
\] Applying now Theorem  \ref{T-compl.} to the present case, we get
\begin{theorem}\label{class.eq.} For any group $A$, the assignment $q \longmapsto \emph{\textsf{Ker}}(q)$ yields a one-to-one
bijection
\[
\bigsqcup_{B < A}\D^{1}(\emph{\T}^l\!\!_{\imath\!_B},  (B,m_{B})) \simeq \emph{\textsf{Fac}}(A)
\]  of sets.
\end{theorem}

\section{COMPARISON WITH SERRE'S NON-ABELIAN COHOMOLOGY}\label{comp.}
\begin{thm} \label{ac.mon.}\em
Let $X$ be a monoid. It is well-known  that to define  an internal monoid object in the category ${^{X}}\!\verb"Sets"$
is the same as to give a left action of $X$ on a monoid. A \emph{(left) action of $X$ on a monoid $B$} is a map
$\star: X \times B \to  B \qquad  (x,b)\longmapsto {x\star b}$, such that
\begin{itemize}
  \item [\rm{(i)}] $\textsf{1}\!_{X}\star b=b;$
  \item [(ii)] $(x_1 x_2)\star b=x_1 \star (x_2 \star b);$
  \item [(iii)] $ x\star \textsf{1}\!_{B}=\textsf{1}\!_{B}$;
  \item [(iv)] $x \star (b_1  b_2)  =(x \star b_1) (x \star b_2)$,\qquad\qquad   $(x,x_1,x_2\in X, \quad b,b_1,b_2 \in B )$.
\end{itemize}
When $X$ acts on a monoid $B$ from the left, we sometimes say that \emph{$B$ is a (left) $X$-monoid}.
If  in the above definition monoids are  replaced by groups, one gets the notion of \emph{left $X$-groups}.
Let us note that in this case, (iii) is a consequence of (iv) by putting $b_1=b_2=\textsf{1}_{B}$.

Note that Conditions (i) and (ii) express the fact that $B$ is a left $X$-set, while Conditions (iii) and (iv)
express the fact that for any fixed $x\in X$, the map
\[
\mathcal{L}_x: B \to B, \qquad b \mapsto x \star b
\]
is a monoid morphism. Hence Conditions (i)--(iv) are equivalent to saying that the assignment $x \to \mathcal{L}_x$  yields a homomorphism
$\Phi: X \to \textsf{End}(B)$ of monoids, where $\textsf{End}(B)$ is  the monoid consisting of all endomorphisms of the $B$.
Because of this, a left action of a monoid $X$ on a monoid $B$ is often called a \emph{left action of $X$ on $B$ by
endomorphisms}.

In analogy with the classical case of groups, any left $X$-monoid $B$ gives rise to the \emph{semiderect product}
${B \rtimes_\Phi X}$, where  $B \rtimes_\Phi X$ is a monoid whose underlying set is the product $B\times X$
equipped with monoid structure given by
\[
(b_1, x_1) \cdot (b_2, x_2) =(b_1 \cdot (x_1 \star b_2), x_1 x_2), \quad\text{and}\quad  \textsf{1}_{B \rtimes_\Phi X} = (\textsf{1}_B,\textsf{1}_X).
\]
Note that the maps
\[
\iota_B: B \to B \rtimes X, \qquad \iota_B(b)=(b,\textsf{1}_X)
\] and
\[
\iota_X: X \to B \rtimes X, \qquad \iota_X(x)=(\textsf{1}_B,x)
\]
are  monoid homomorphisms.
\end{thm}

\begin{thm}\label{nccm}\em
Let $X$ be a monoid and  let $B$ be a left $X$-monoid. The \emph{zeroth cohomology set of $X$ with coefficients in $B$}
is the subset  of elements of $B$ fixed by $X$:
\[
\mathcal{H}^0(X,B)={}^X\!B=\Big\{b \in B\mid x\star b=b \quad \forall x \in X\big\}.
\]
The set $\mathcal{H}^0(X,B)$ is a submonoid of $B$ by Conditions \ref{ac.mon.} (iii) and (iv).

A function  $q: X \to B$ is called  a \emph{1-cocycle } if
\[
q(x_1  x_2)=q(x_1) (x_1 \star q(x_2)), \qquad (\forall x_1,x_2\in X).
\]
  Write $\mathcal{Z}^{1}(X, B)$ for the set of 1-cocycles $X\to B$.
Quite obviously  the map
$$
0_{X,B}: X \to B, \quad  x \longmapsto \textsf{1}_B
$$
is a 1-cocycle. This map enables us to regard  $\mathcal{Z}^{1}(X, B)$ as a pointed set.

Two 1-cocyles $q, q': X \to B$ are called \emph{equivalent} if there exists an invertible element $b_0\in U(B)$
such that $q(x) (x \star b_0)=b_0  q'(x)$ for all $x \in X$. This provides an equivalence relation on $\mathcal{Z}^{1}(X, B)$.
The proof proceeds as in the case of groups (e.g., \cite{S}). We need to show only that $x\star b_0$ is invertible
for any $b_0\in U(B), x\in X$,  and $(x\star b_0)^{-1}=x\star b_0^{-1}$. But since $b_0\in U(B)$, it follows that
there exists an element $b_0^{-1}$ with $b_0b_0^{-1}=b_0^{-1}b_0=\textsf{1}_B$. Now we have
$$
\textsf{1}_B=x\star \textsf{1}_B=x \star (b_0 b_0^{-1})=(x \star b_0)(x \star b_0^{-1}),
$$
so $(x \star b_0)(x \star b_0^{-1})=\textsf{1}_B$. In a similar manner one
can  prove that $(x \star b_0^{-1})(x \star b_0)=\textsf{1}_B$.
Therefore, $(x\star b_0)^{-1}=x\star b_0^{-1}$.

The resulting set of equivalence classes of 1-cocycles is called the \emph{first non-abelian 1-cohomology
set of $X$ with coefficients in $B$} and is denoted by $\mathcal{H}^1(X,B)$.
It is a pointed set where the  distinguished point is the equivalence class of the  maps $0_{X,B}$.

\begin{remark}\label{serre}\em In the case where $X$ is  a group and  $B$ is a left $X$-group, the group $\mathcal{H}^0(X,B)$
and the pointed set $\mathcal{H}^1(X,B)$ coincide with the ones  introduced by Serre in \cite{S}.
\end{remark}

Let $X$ be a monoid, $B$ a left $X$-monoid and $B \rtimes X$ the corresponding semidirect product.
Direct inspection shows that the projection $p_B: B \times X \to B$ satisfies Conditions (ZL1)--(ZL3), and hence
defines the structure of a $\T^l\!\!_{\iota_B}$-algebra on $(B,m_B)$. Thus, the triple $((B,m_B), p_B)$ is an
object of the category $(^{B}\verb"Sets")^{\T^l\!\!_{\iota\!_B}}$ and hence (see \ref{des.coh.})
$\mathcal{Z}^1(\T^l\!\!_{\iota\!_B}, ((B,m_B), p_B))$  becomes a pointed set with point the $p_B$.
Then $\D^1(\T^l\!\!_{\iota\!_B}, ((B,m_B), p_B))$ is pointed set with  point $[p_B]$.

\medskip

Recall that given a monoid $M$, the \emph{opposite} monoid $M^{\text{op}}$ has the same underlying
set and identity element as $M$, and its multiplication is defined by $m \cdot^{\text{op}} m'=m'\cdot m$.

\begin{remark}\label{opp.}\em
If $B$ is a left $X$-monoid, then $B^{\text{op}}$ becomes a left $X$-monoid
via $x \star b^{\text{op}}=(x \star b)^{\text{op}}$, which is  called the \emph{opposite} of  the $X$-monoid $B$.
\end{remark}

\begin{proposition}\label{non.ab.coh.m}  In the situation described above, the assignment
$$
(q: B \times X \to B)\longmapsto (\ol q=q \cdot \iota_X: X \to B)
$$
yields an isomorphism
$$
\mathcal{Z}^1(\emph{\T}^l\!\!_{\iota\!_B}, ((B,m_B), p_B))\simeq \mathcal{Z}^1(X,B^{\!\text{op}})
$$ of pointed sets. Here, $B^{\!\text{op}}$ has a left $X$-monoid structure as given  in Remark \ref{opp.}.
\end{proposition}
\begin{proof} According to Theorem \ref{T-str} and Remark \ref{remark.str}(b),
to give a $\T^l\!\!_{\iota\!_B}$-algebra structure on $(B, m_B)\in ^B\verb"Sets"$ is to give a map
$q: B \times X \to B$ of left $B$-sets such that it sends $\textsf{1}_{B \rtimes X}=(\textsf{1}_B,\textsf{1}_X)$
to $\textsf{1}_B$ and makes the diagram
\begin{equation}\label{cocycle}
\xymatrix @R=.4in @C=.3in{(B \times X)\times (B \times X) \ar[rr]^-{m_{_{B \rtimes X}}} \ar[d]_{(B \times X)\times q}&& B \rtimes X \ar[dd]^q\\
(B \times X)\!\times B \ar[d]_{(B \times X)\times \iota_B}&&\\
(B \times X)\times(B \times X)\ar[r]_-{m_{_{B\rtimes X}}}& B\times X\ar[r]_-{q}&B}\end{equation} commutative.
Suppose first that $q: B\times X \to B$ is such a map.  Since $q(\textsf{1}_B, \textsf{1}_X)=q(\textsf{1}_{B \rtimes X})=\textsf{1}_B$,
it follows that $\ol q (\textsf{1}_X)=(q \cdot \iota_X)(\textsf{1}_X)=q(\textsf{1}_B, \textsf{1}_X)=\textsf{1}_B$.

\noindent
Next, since
$$
m_{_{B \rtimes X}}((b_1, x_1), (b_2, x_2) )=(b_1 (x_1 \star b_2),x_1 x_2)
$$
for all $b_1,b_2 \in B$ and all $x_1,x_2 \in X$, chasing $((b_1, x_1),(b_2, x_2))$ around the diagram (\ref{cocycle}) gives the
equality
$$
q(b_1 (x_1 \star b_2), x_1 x_2)=q(b_1 (x_1 \star q(b_2,x_2)),x_1).
$$
Since $q$ is assumed to be a map of left $B$-sets and since clearly $b_1 (x_1 \star b_2)\in B$ and  $b_1 (x_1 \star q(b_2,x_2))\in B$, it can be written as
$$
b_1 (x_1 \star b_2) q(\textsf{1}_B, x_1 x_2)=b_1 (x_1 \star q(b_2,x_2)) q(\textsf{1}_B,x_1).
$$
Moreover, since $x \star \textsf{1}_B=\textsf{1}_B$ for all $x \in X$, putting $b_1=b_2=\textsf{1}_B$ gives
$$
\ol q (x_1 x_2)=q(\textsf{1}_B, x_1 x_2)=(x_1 \star q(\textsf{1}_B,x_2))q(\textsf{1}_B,x_1)=(x_1 \star\ol q(x_2)) \ol q(x_1),
$$
proving that $\ol q \in \mathcal{Z}^1(X,B^{\text{op}})$.

Conversely, suppose that $q' \in \mathcal{Z}^1(X,B^{\text{op}})$ is an arbitrary 1-cocycle. We claim
that the composite
\[
\widehat{q'}: B\times X \xr{B \times  q'} B \times B \xr{m_B} B,\qquad (b,x)\longmapsto bq'(x)
\]
is a $\T^l\!\!_{\iota\!_B}$-algebra structure on $((B,m_B), p_B)$. Indeed, since $q'(\textsf{1}_X)=\textsf{1}_B$,
it follows that $\widehat{q'}(\textsf{1}_B,\textsf{1}_X)=\textsf{1}_B q'(\textsf{1}_X)=\textsf{1}_B$.
Thus, $\widehat{q'}(\textsf{1}_X)=\textsf{1}_B$. Moreover, $\widehat{q'}$ is a map of left $B$-sets, since

\[
\begin{split}
\widehat{q'}(b_1 \cdot (b_2,x))={q'}(b_1 b_2,x)&=(b_1 b_2) q' (x)\\
&=b_1 (b_2  q' (x))= b_1  \widehat{q'} (b_2, x),\qquad (b_1,b_2 \in B, x \in X).
\end{split}
\]
Next, we calculate
\begin{equation}\label{non.ab.eq.}
\begin{split}
\widehat{q'}(b_1 (x_1\star b_2), x_1 x_2)&=b_1 (x_1\star b_2) q'(x_1 x_2)\\&=
b_1 (x_1\star b_2) (x_1 \star q'(x_2)) q' (x_1).
\end{split}
\end{equation}
Here, the second equality holds since $q'$ is 1-cocycle.
Then since
\begin{align*}
&\widehat{q'}(b_1 (x_1 \star \widehat{q'}(b_2,x_2)),x_1)=\\
&=b_1 (x_1 \star \widehat{q'}(b_2,x_2)) q'(x_1)& \text{by definition of}\quad \widehat{q'}\\
&=b_1 (x_1 \star (b_2 q'(x_2)) q'(x_1)& \text{by definition of}\quad \widehat{q'}\\
&=b_1 (x_1\star b_2)(x_1\star q'(x_2) q'(x_1)& \text{by Condition (ii) of the action}\quad \star\\
&=\widehat{q'}(b_1 (x_1\star b_2), x_1 x_2), & \text{by} \qquad \eqref{non.ab.eq.}
\end{align*} it follows that $\widehat{q'}$ makes Diagram (\ref{cocycle}) commutative. This proves
that $\widehat{q'}$ gives a $\T^l\!\!_{\iota\!_B}$-algebra structure on $(B, m_B)$. It is easily
verified that $\widehat{\ol q}=q$ and  $\ol{\widehat{q'}}=q'$.   Moreover,
$\ol{p}_B: X\xr{\iota \!_X} B \times X \xr{p_B} B$ is the map $0_{X,B}: X \to B$. Hence
the assignment $q \longmapsto \ol q$ yields a bijection
$$\mathcal{Z}^{1}(\T^l\!\!_{\iota\!_B},  (B,m_B))\simeq \mathcal{Z}^{1}(X,B^{op})$$ of pointed sets.
\end{proof}

\begin{theorem}\label{non.ab.cohomology} Let  $B$ be a left $X$-monoid. The assignment
\[
(q: B \times X \to B)\longmapsto (\ol q=q \cdot \iota_X: X \to B)
\]
yields an isomorphism $$\D^1(\emph{\T}^l\!\!_{\iota\!_B}, ((B,m_B), p_B))\simeq \mathcal{H}^1(X,B^{op})$$
of pointed sets. Moreover, there is an isomorphism of groups
$$\D^0(\emph{\T}^l\!\!_{\iota\!_B}, ((B,m_B), p_B))\simeq \mathcal{H}^0(X,B^{op}).$$
\end{theorem}
\begin{proof} By Proposition \ref{iso-T-str}, two 1-cocycles $q, q' \in \mathcal{Z}^{1}(\T^l\!\!_{\iota \!_B},  (B,m_B))$
are equivalent if and only if there is an invertible element $b_0 \in B$ such that
\begin{equation}\label{non.ab.eq.1}
q(b,x) b_0=q\,'((b,x) \cdot \iota_B(b_0)), \qquad (\forall (b,x) \in B \times X). \end{equation} Since
$q(b,x) b_0=b \, \ol q(x) \, b_0$ and
$$
q\,'((b,x) \cdot \iota_B(b_0))=q\,'((b,x) \cdot (b_0, \textsf{1}_X))=q\,'(b(x \star b_0),x)=b(x \star b_0)\,\ol q\,'(x),
$$
the equation  (\ref{non.ab.eq.1}) holds if and only if
$$b \, \ol q(x) \, b_0=b(x \star b_0)\,\ol q\,'(x), \qquad (\forall (b,x) \in B \times X)
$$
or equivalently, if and only if
$$
\ol q(x) \, b_0=(x \star b_0)\,\ol q\,'(x), \qquad (\forall  x \in X).
$$
The last equation expresses  the fact that the 1-cocycles $\ol q, \ol q\,' \in \mathcal{Z}^{1}(X,B^{op})$
are equivalent, so  the map that takes $[q]$ to $[\ol{q}]$ yields an isomorphism
\[
\D^1(\T^l\!\!_{\iota\!_B}, ((B,m_B), p_B)) \to \mathcal{H}^1(X,B^{\text{op}})
\]
of pointed sets by  Proposition \ref{non.ab.coh.m}.

To show the second part of the theorem, recall first that $\D^0(\T^l\!\!_{\iota \!_B}, ((B,m_B),p_B))$
is the group of automorphisms of $((B,m_B),p_B)$ in $(^B\verb"Sets")^{\T^l\!\!_{\iota\!_B}}$. However
to give an isomorphism of $((B,m_B),p_B)$ to itself in $(^B\verb"Sets"){\T^l\!\!_{\iota \!_B}}$ is the
same as to find  an invertible  element $b_0 \in B$ making the diagram commutative
\[
\xymatrix @R=.3in @C=.4in{B \times X \ar[r]^-{-\cdot \iota_B(b_0)} \ar[d]_{p_B}& B \times X \ar[d]^{p_B}\\
B \ar[r]_{-\cdot b_0}& B} \qquad \xymatrix @R=.3in @C=.6in{(b,x) \ar[r] \ar[d]& (b,x)\cdot\iota_B(b_0)=(b (x\star b_0), x)  \ar[d]\\
b \ar[r]& b b_0\stackrel{?}=b (x\star b_0).}
\]
It follows that the diagram is commutative if and only if
\[
b b_0=b (x\star b_0), \qquad\quad  (\forall b\in B, x\in X)
\]
and hence if and only if $b_0=x \star b_0$ for all $x\in X$, or equivalently, if and only if  $ b_0 \in {}^{X}\!B$.
This proves that $\D^0(\T^l\!\!_{\iota \!_B}, ((B,m_B),p_B))\simeq \mathcal{H}^0(X, B^{\text{op}})$ as groups.
\end{proof}
\end{thm}

\begin{theorem}\label{non.ab.coh.gr.} Let  $B$ be a left $X$-group, where $X$ is a group. The assignment
$$[q] \longmapsto [(q \iota_X)^{-1}],$$ where the map $(q \iota_X)^{-1} : X\to B$ is defined by
$(q \iota_X)^{-1}(x)=((q \iota_X)(x))^{-1}$, yields an isomorphism $$\D^1(\emph{\T}^l\!\!_{\iota \!_B},
((B,m_B), p_B))\simeq \mathcal{H}^1(X,B)$$ of pointed sets. Moreover, there is an isomorphism of groups
$$\D^0(\emph{\T}^l\!\!_{\iota \!_B}, ((B,m_B), p_B))\simeq \mathcal{H}^0(X,B).$$
\end{theorem}
\begin{proof} If $f: B \to B'$ is an isomorphism of groups, then the composite
$$X\times B' \xr{X \times f ^{-1}}  X\times B \xr {\star} B \xr{f} B'$$ defines a unique left
$X$-group structure on $B'$ for which $f$ becomes an isomorphism of left $X$-groups. Moreover, $f$ yields
a bijection $\mathcal{H}^1(X,B)\simeq \mathcal{H}^1(X,B')$ of pointed sets and an isomorphism
$\mathcal{H}^0(X,B)\simeq \mathcal{H}^0(X,B')$ of groups (see \ref{des.coh.}).

In particular, each  group $B$ is isomorphic to its opposite via the isomorphism $B \xr{b \to b^{-1}} B^{\text{op}}$, so
the result follows from Theorem \ref{non.ab.cohomology}.
\end{proof}

Combining Theorems \ref{T-compl.} and \ref{non.ab.coh.gr.} gives the following result, which can be seen
as a generalization of \cite[Chapter 17, Proposition 33]{Dummit}.

\begin{theorem}\label{non.ab.coh.gr.ad} For a group $X$ and a left $X$-group $B$, there is
an isomorphism $$\pi_0 (\mathcal{C}_{\iota\!_B}/\!/B)\simeq \mathcal{H}^1(X,B)$$ of pointed sets.
\end{theorem}

Finally we present  examples of the computation of $1$-descent cohomology sets.

\bigskip

For a positive integer $n$, let $S_n$ be the symmetric group of degree $n$ and let $A_n$ be the alternating subgroup of $S_n$, i.e.,
the subgroup consisting of even permutations.

\begin{thm}{\bf Example.} \label{example.1}\em
Consider $S_2$ as a subgroup in $S_3$ by letting $3$ to be a fixed point and write
$\iota_{2/3}: S_2 \to S_3$ for the corresponding inclusion. It is well known that $S_2$ is not normal in $S_3$. Since the subgroups of $S_3$ are
\begin{itemize}
\item the trivial subgroup $\{()\}$,

\item isomorphic to $S_2$: $\{(), (12)\},\, \{(), (23)\},\, \{(), (13)\}$,

\item $A_3$: $\{(), (123),\, (132)\}$, and

\item the whole group $S_3$,
\end{itemize}
it follows that $A_3$ is the unique complement to $S_2$ in $S_3$. Thus
\[
\mid\D^1(\T^l\!_{\iota_{2/3}}, S_2)\mid=1.
\]
The corresponding $\T^l\!_{\iota_{2/3}}$-algebra structure on $S_2$ is the map
\[
(q: S_3\to S_2)(\sigma)=\begin{cases}
(), \quad  &\text{if}\quad \sigma\in A_3;\\
(12),\quad &\text{if}\quad \sigma\in S_3 \setminus A_3.
\end{cases}
\]

\end{thm}

\begin{thm}{\bf Example.} \label{example.2}\em
Let  $\iota_n: A_n \to S_n$ for some $n\geq 3$. Since $S_n/A_n \simeq S_2$, each  complement $X$ to $A_n$ in $S_n$ must be a subgroup  of $S_n$ of order 2. However  since any such  subgroup has  the form $X=((1), \sigma)$,  where $\sigma$ is a permutation of order 2, and  since the order of a permutation written in the disjoint cycle form is the least common multiple
of the lengths of the cycles, it follows that $\sigma$
must be a product of $m$ disjoint transpositions, for some $m \in \mathbb{N}$. Since $A_n X= S_n$,
$\sigma$ (and hence also $m$) must be odd. Next,
since two permutations in $S_n$ are conjugate if and only if they have the same cycle structure, it follows that
$X= ((1), \sigma)$ and $X' = ((1), \sigma')$ are conjugate in $S_n$ if and only if  both $\sigma$ and $\sigma'$
are products of $m$ disjoint 2-cycles for some odd $m$. A straightforward calculation now gives
\[
\mid\D^1(\T^l\!\!_{\iota_n}, A_n)\mid=
\begin{cases}
1, \quad &\text{if}\quad n=3;\\
k, \quad &\text{if}\quad n\in\{ 4k, 4k+1\};\\
k+1, \quad &\text{if}\quad n\in\{4k+2, 4k+3 \}.
\end{cases}
\]
\end{thm}

\begin{thm}{\bf Example.} \label{example.3}\em Let $f:K \to M$ and $g:L \to M$ be group homomorphisms and
$K \times_M L=\{(k,l)\in K \times L:f(k)=g(l)\}$ be the corresponding pullback. Write $\iota$ for the
group homomorphism $B=\textsf{Ker}(g) \to K \times_M L$ that takes $l \in \textsf{Ker}(g)$ to $(\textsf{1}_{K},l)$ and
write $\textit{Hom}^{f}_{g}(K,L)$ for the set
\[
\big\{\text{all homomorphisms}\,\,K \to L \,\,\text{such that}\,\, gh =f\big\}.
\]
It is easy to see that for any $h \in \textit{Hom}^{f}_{g}(K,L)$ and any $l_0 \in \textsf{Ker}(g)$, the map
\[
h^{l_0}=l_0 h  l^{-1}_0=:k \longmapsto l_0 h(k)  l^{-1}_0, \qquad\quad   (k \in K)
\]
 lies again in $\textit{Hom}^{f}_{g}(K,L)$.
Thus the group $\textsf{Ker}(g)$ acts on $\textit{Hom}^{f}_{g}(K,L)$. The quotient of $\textit{Hom}^{f}_{g}(K,L)$
by the action of $\textsf{Ker}(g)$ is denoted by $\textsf{Hom}^{f}_{g}(K,L)$.

For any $h \in \textit{Hom}^{f}_{g}(K,L)$, we write $q_h$ for the map $K \times_ML \to L$
given by $$(k,l) \longmapsto l\cdot h(k^{-1}).$$ Since for any $(k,l)\in K\times_ML$, $f(k)=g(l)$
and hence $g(l)\cdot f(k^{-1})=\textsf{1}_M$, it follows that $g(l \cdot h(k^{-1}))=g(l)\cdot (gh)(k^{-1})=
g(l)\cdot f(k^{-1}))=\textsf{1}_M$. Thus, $l \cdot h(k^{-1}) \in \textsf{Ker}(g)$. Therefore, $q_h$
can be seen as a map $K \times_ML \to \textsf{Ker}(g)$. We then claim that $q_h \in \mathcal{Z}^{1}(\T^l\!\!_\iota,  (B,m_B))$.
To show the claim, it is enough by Proposition \ref{T-str.00} to show that $q_h$ satisfies (ZL2) and (ZL3).
For (ZL2), we have:
\[
\begin{split}
q_h(\iota(l)\cdot (k,l_1))=\\
=& q_h((\textsf{1}_K,l)\cdot (k,l_1))=q_h(k,ll_1)\\
=& (ll_1)\cdot h(k^{-1})= l\cdot(l_1\cdot h(k^{-1}))\\
  =&l \cdot q_h(k,l_1),\qquad\qquad \qquad\qquad  (l \in L, (k, l_1)\in K \times_ML).
\end{split}\]

\noindent To prove that $q_h$ satisfies (ZL3), we make the following calculation:
\begin{align*}
q_h((k_1,l_1)\cdot \iota q_h(k_2,l_2))=\\
&=q_h((k_1,l_1)\cdot (\textsf{1}_K,l_2 \cdot h(k^{-1}_2)))& \text{by definition of}\quad q_h\\
&=q_h((k_1,(l_1l_2) \cdot h(k^{-1}_2)))& \\
&=(l_1l_2) \cdot h(k^{-1}_2)\cdot h(k^{-1}_1)& \text{by definition of}\quad q_h\\
&=(l_1l_2) \cdot h(k^{-1}_2 k^{-1}_1)& \\
&=(l_1l_2)\cdot h((k_1k_2)^{-1})\\
&=q_h(k_1k_2,l_1l_2)&\text{by definition of}\quad q_h\\
&=q_h((k_1,l_1)\cdot (k_2,l_2)).
\end{align*}
Summarizing, we find that the assignment $h \longmapsto q_h$ yields a map
$$ \textit{Hom}^{f}_{g}(K,L) \to \mathcal{Z}^{1}(\T^l\!\!_\iota,  (B,m_B))$$
of sets. We write $\phi$ for this map. It follows from the very definition that $\phi$ is injective.
We next show that $\phi$ induces a map
$$\ol{\phi} : \textsf{Hom}^{f}_{g}(K,L) \to \D^{1}(\T^l\!\!_\iota,  (B,m_B)).$$

\noindent Indeed, suppose that $h, h' \in {Hom}^{f}_{g}(K,L)$ are such that their classes coincide in $\textsf{Hom}^{f}_{g}(K,L)$.
Then there exists an element $l_0\in B\subseteq L$ with $h'=l_0hl_0^{-1}$ (i.e.,
$h'(k)=l_0h(k)l_0^{-1}$ for all $k \in K$). Then for all $(k,l) \in K \times_ML$, we have:
\[
\begin{split}
q_{h'}(k,l)\cdot l_0=lh'(k^{-1})l_0=\\
=&ll_0h(k^{-1})l_0^{-1}l_0=ll_0h(k^{-1})\\
=&q_h(k,ll_0)=q_h((k,l)\cdot \iota(l_0)).
\end{split}\] Here, the first and fourth equality follow from the definition of $q_h$.
It then follows from Proposition \ref{iso-T-str} that $[q_{h}]=[q_{h'}]$ in $\D^{1}(\T^l\!\!_\iota,  (B,m_B))$.

We now investigate under which conditions $\phi$ is bijective. To do this, we consider an
arbitrary $q \in \mathcal{Z}^{1}(\T^l\!\!_\iota,  (B,m_B))$. If $(k,l_1), (k,l_2)\in X=\textsf{Ker}(q)$,
then $f(k)=g(l_1)=g(l_2)$ and hence $g(l_1l_2^{-1})=\textsf{1}_M$.
Thus, $(\textsf{1}_K, l_1l_2^{-1}) \in \iota(B)$. Since $X$ is a subgroup of the group $K \times_ML$,
$(k,l_2)^{-1}=(k^{-1}, l_2^{-1})\in X$. Then
\[
(k,l_1) \cdot (k^{-1}, l_2^{-1})=(\textsf{1}_K, l_1l_2^{-1})\in X.
\]
But since $X$ is a complement to $\iota(B)$ in $K \times_ML$ (see, Proposition \ref{T-compl.01}),
$\iota(B)\cap X=(\textsf{1}_K, \textsf{1}_L)$, implying that $l_1l_2^{-1}=\textsf{1}_L$
and hence $l_1=l_2$. It follows that for any $k \in K$, there exists at most one $l \in L$ with $(k,l) \in X$.
Suppose now that $f(K)\subseteq g(L)$. Then for every $k \in K$, there exists  $l \in L$ with $(k,l) \in K \times_ML$.
It follows -- since $\iota(B)X=K \times_ML$ -- that $(k,l)=(\textsf{1}_K, l')\cdot (k, l_k)$ with $(k, l_k)\in X$.
Thus, for every $k \in K$, there exists a (necessarily unique) element $l_k$ such that $(k, l_k)\in X$
and hence to any $q \in \mathcal{Z}^{1}(\T^l\!\!_\iota,  (B,m_B))$ one can associate a map
 $$h_q: K \to B, \,\, h_q(k)=l_k.$$
Note that if $(k_1,l_{k_1}), (k_2,l_{k_2})\in X$, then $(k_1k_2, l_{k_1}l_{k_2})=(k_1,l_{k_1})\cdot (k_2,l_{k_2})\in X$,
since $X$ is a subgroup of the group $K \times_ML$. It follows that $h_q(k_1k_2)=l_{k_1}l_{k_2}$ and hence $h_q$
is a homomorphism of groups. Consequently, when, as  we  henceforth  suppose, $f(K)\subseteq g(L)$,
one may define a map
$$\phi':\mathcal{Z}^{1}(\T^l\!\!_\iota,  (B,m_B)) \to  \textit{Hom}^{f}_{g}(K,L)$$
by $\phi'(q)= h_q$. We claim that $\phi'$ is the inverse of $\phi$. Indeed, if $h \in {Hom}^{f}_{g}(K,L)$,
then for all $(k,l) \in K \times_ML$, $h_{q_h}(k)=l_k$, where $l_k\in L$ is the unique element such that $q_h(k,l_k)=\textsf{1}_L$.
But since $q_h(k,l)=lh(k^{-1})$, to say that $h_{q_h}(k)=l_k$ is to say that $l_kh(k^{-1})=\textsf{1}_L$,
or, equivalently, that $l_k=h(k)$. Consequently, $h_{q_h}(k)=h(k)$ for all $k \in K$. Thus $\phi'\phi=\textsf{Id}$.
Conversely, if $q\in \mathcal{Z}^{1}(\T^l\!\!_\iota,  (B,m_B))$, then for any $(k,l) \in K \times_ML$,
$q_{h_q}(k,l)=lh_q(k^{-1})=ll_k^{-1}$, where $l_k \in L$ is the unique element such that $q(k,l_k)=\textsf{1}_L$.
Since $(k,l)\in K \times_ML$ and $(k,l_k) \in \textsf{Ker}(q)$, it follows that $f(k)=g(l)=g(l_k)=\textsf{1}_M$.
Then $g(ll_k^{-1})=\textsf{1}_M$ and thus $(\textsf{1}_K, ll_k^{-1})\in \iota(B)$.
As $$(k,l)=(\textsf{1}_K, ll_k^{-1})\cdot (k, l_k)=\iota(ll_k^{-1})\cdot (k, l_k)$$ in $K \times_ML$,
it follows that $q(k,l)=ll_k^{-1}$. Thus $q=q_{h_q}$, proving that $\phi\phi'=\textsf{Id}$. Consequently,
$\phi$ is a bijection. We next show that $\ol{\phi}$ is also a bijection. Indeed, considering the commutative
diagram
\[
\xymatrix{  \textit{Hom}^{f}_{g}(K,L) \ar[d]\ar[r]^{\phi} & \mathcal{Z}^{1}(\T^l\!\!_\iota,  (B,m_B))\ar[d]\\
\textsf{Hom}^{f}_{g}(K,L) \ar[r]_(.40){\ol{\phi}}& \D^{1}(\T^l\!\!_\iota,  (B,m_B)), }
\] where the vertical maps are the canonical surjections, it suffices to show that $\ol{\phi}$ is injective.
So suppose that $h_1,h_2 \in \textit{Hom}^{f}_{g}(K,L)$ are such that $[q_{h_1}]=[q_{h_2}]$ in $\D^{1}(\T^l\!\!_\iota,  (B,m_B))$.
Then there exists an element $l_0 \in B$ with
\[
q_{h_1}(k,l)l_0=q_{h_2}((k,l)\cdot(\textsf{1}_K,l_0))=q_{h_2}(k,ll_0), \qquad  \text{for all}\,\, (k,l)\in K\times_ML
\]
and hence $lh_1(k^{-1})l_0=ll_0 h_2(k^{-1})$ for all $(k,l)\in K\times_ML$. Since for any $k \in K$ there
exists at least one $l \in L$ with $(k,l)\in K\times_ML$, it follows that $h_2=h_1^{l_0}$ and hence $[h_1]=[h_2]$
in $\textsf{Hom}^{f}_{g}(K,L)$. Consequently, $\ol{\phi}$ is bijection.

\medskip
Summing up, we have established:
\medskip

\emph{Let $f:K \to M$ and $g:L \to M$ be group homomorphisms, $B=\emph{\textsf{Ker}}(g)$ and $\iota:B \to K \times_M L$
the group homomorphism defined by $\iota(l)= (\textsf{1}_{K},l)$. Then the assignment $h \to q_h$ yields bijections
$$ \textit{Hom}^{f}_{g}(K,L) \ \simeq \mathcal{Z}^{1}(\emph{\T}^l\!\!_\iota,  (B,m_B))$$
and
$$\emph{\textsf{Hom}}^{f}_{g}(K,L) \to \D^{1}(\emph{\T}^l\!\!_\iota,  (B,m_B)).$$}
\end{thm}

\begin{thm}{\bf Remark.} \em
Recently \cite{B-M} and \cite{Me-T} present some interesting applications of our results.
\end{thm}

\end{document}